\documentclass[10pt,reqno]{amsart}
\usepackage{amssymb,amsmath}
\usepackage{amscd}
\usepackage[top=3.2cm, bottom=3.2cm, left=3.4cm, right=3.4cm]{geometry}

\theoremstyle{plain}
\newtheorem{theorem}{Theorem}[section]                    
\newtheorem{proposition}[theorem]{Proposition}            
\newtheorem{corollary}[theorem]{Corollary}                
\newtheorem{lemma}[theorem]{Lemma}
\newtheorem{claim}[theorem]{Claim}
\theoremstyle{definition}
\newtheorem{remark}[theorem]{Remark}

\newtheorem{definition}[theorem]{Definition}

\def\us#1_#2{\underset{#2}{#1}}
\def\os#1^#2{\overset{#2}{#1}}
\def\i<#1>{\langle {#1} \rangle}
 
\newcommand{\alg}{{\mathord{\rm alg}}}

\newcommand{\red}{{\mathord{\rm red}}}

\DeclareMathOperator{\Ad}{Ad}

\DeclareMathOperator{\id}{id}

\makeatletter 
\@tempcnta\z@
\loop\ifnum\@tempcnta<26
\advance\@tempcnta\@ne
\expandafter\edef\csname frak\@Alph\@tempcnta\endcsname{\noexpand\mathfrak{\@Alph\@tempcnta}}
\expandafter\edef\csname l\@Alph\@tempcnta\endcsname{\noexpand\mathbb{\@Alph\@tempcnta}}
\expandafter\edef\csname cal\@Alph\@tempcnta\endcsname{\noexpand\mathcal{\@Alph\@tempcnta}}
\expandafter\edef\csname rm\@Alph\@tempcnta\endcsname{\noexpand\mathrm{\@Alph\@tempcnta}}
\expandafter\edef\csname bf\@Alph\@tempcnta\endcsname{\noexpand\mathbf{\@Alph\@tempcnta}}
\repeat
\makeatother
 
\makeatletter 
\@tempcnta\z@
\loop\ifnum\@tempcnta<26
\advance\@tempcnta\@ne
\expandafter\edef\csname frak\@alph\@tempcnta\endcsname{\noexpand\mathfrak{\@alph\@tempcnta}}
\repeat
\makeatother

\newcommand{\dotimes}{\mathbin{\dot{\otimes}}}
\newcommand{\ddotimes}{\mathbin{\ddot{\otimes}}}

\allowdisplaybreaks[4]
\begin{document}
\title[$KK$-equivalence for amalgamated free product]{$KK$-equivalence for amalgamated free product $\rmC^*$-algebras }
\author{Kei Hasegawa}
\address{Graduate~School~of~Mathematics, Kyushu~University, Fukuoka~819-0395, Japan}
\email{k-hasegawa@math.kyushu-u.ac.jp}
\begin{abstract}
We prove that {\it any} reduced amalgamated free product $\rmC^*$-algebra is $KK$-equivalent to the corresponding full amalgamated free product $\rmC^*$-algebra. The main ingredient of its proof is Julg--Valette's geometric construction of Fredholm modules with Connes's view for representation theory of operator algebras.
\end{abstract}
\maketitle
\section{Introduction}\renewcommand{\thetheorem}{\Alph{theorem}}

In \cite{Cuntz-82}\cite{Cuntz-83} Cuntz gave a strategy of computing the $K$-theory of the reduced $\rmC^*$-algebra $\rmC^*_\red (\Gamma)$ of a given discrete group $\Gamma$. The strategy consists of two parts: 
\begin{itemize}
\item[(1)] proving that the canonical surjection $\lambda : \rmC^*(\Gamma) \to \rmC^*_\red (\Gamma)$ (where $\rmC^*(\Gamma)$ denotes the full $\rmC^*$-algebra of $\Gamma$) gives a $KK$-equivalence, that is, has an inverse in $KK$-theory, and 
\item[(2)] computing the $K$-theory of $\rmC^*(\Gamma)$.
\end{itemize} 
In fact, usual computations in $K$-theory are made by establishing suitable exact sequences, and the full group $\rmC^*$-algebra $\rmC^*(\Gamma)$ is easier to handle than the reduced one $\rmC^*_\red(\Gamma)$.
By the strategy, Cuntz indeed gave a much simpler proof of celebrated Pimsner--Voiculescu's result of the $K$-theory of $\rmC^*_\red(\mathbb{F}_n)$ (\cite{Pimsner-Voiculescu-82}).
Then Julg and Valette \cite{Julg-Valette} achieved part (1) of the strategy when $\Gamma$ acts on a tree with amenable stabilizers.
In the direction, Pimsner gave in \cite{Pimsner} an optimal result, but his strategy looks different from Cuntz's one.  

\medskip
It is very natural (at least for us) to try to adapt Cuntz's strategy to amalgamated free product $\rmC^*$-algebras.
Part (2) of the strategy was achieved by Thomsen \cite{Thomsen} under a very weak assumption.
Thus, our main problem is part (1) of the strategy.
It was Germain \cite{Germain-96}\cite{Germain-97} who first tried to examine the strategy for plain free product $\rmC^*$-algebras,
and he obtained the desired $KK$-equivalence result for plain free product $\rmC^*$-algebras of nuclear $\rmC^*$-algebras.
Following Germain's idea in \cite{Germain-96}\cite{Germain-97-2} we recently proved in \cite{Hasegawa} (also see \cite{Germain-Sarr}) that the canonical surjection onto a given reduced amalgamated free product $\rmC^*$-algebra from the corresponding full one gives a $KK$-equivalence under the assumption that every free component is ``strongly relative nuclear" against the amalgamated subalgebra.
This was is a byproduct of our attempt to seek for a suitable formulation of ``relative nuclearity" for inclusions of $\rmC^*$-algebras. 

\medskip
In this paper we adapt, unlike \cite{Germain-96}\cite{Germain-97-2}\cite{Germain-Sarr}\cite{Hasegawa}, Julg--Valette's geometric idea to the problem, and  establish the optimal $KK$-equivalence result for amalgamated free product $\rmC^*$-algebras.
We emphasize that the core part of the proof is very simple and just 3 pages long (though this paper is rather self-contained).
To state our main result precisely, let us give a few terminologies.
Let $\{ (B \subset A_i , E^{A_i}_B ) \}_{i \in \calI }$ be a countable family of quasi-unital inclusions of separable $\rmC^*$-algebras with nondegenerate conditional expectations from $A_i$ onto $B$.
Here $B \subset A_i$ is quasi-unital if $BA_i B$ is norm-dense in $A_i$, and also $E^{A_i}_B : A _i \to B$ is nondegenerate if the associated GNS representation is faithful.
Let $(A,E) = \bigstar_{B, i \in \calI} (A_i, E^{A_i}_i)$ be the reduced amalgamated free product and we call $A$ the reduced amalgamated free product $\rmC^*$-algebra. Also, let $\mathfrak{A} = \bigstar_{B, i \in \calI} A_i$ be the full amalgamated free product $\rmC^*$-algebra. With the notation we will prove the following: 

\begin{theorem}\label{thm-KK}
The canonical surjection $\lambda : \mathfrak{A} \to A$ gives a $KK$-equivalence without any extra assumption.  
\end{theorem}

The proof is done by translating the ``geometric" construction of Fredholm modules due to Julg--Valette (and its quantum group analog due to Vergnioux \cite{Vergnioux}) into a $\rmC^*$-algebraic language following Connes's view of correspondences.
This is similar to our previous work \cite{Hasegawa} on relative nuclearity.
More precisely, we will easily prove that the canonical surjection $\lambda$ gives a $KK$-subequivalence like Julg--Valette \cite{Julg-Valette} and Vergnioux \cite{Vergnioux}.
Then we will directly prove that $\lambda$ indeed gives a $KK$-equivalence.
The latter is unnecessary in the amenable (quantum) group case \cite{Julg-Valette}\cite{Vergnioux} thanks to the existence of counits,
and is the most original part of the present paper. 
As a bonus of the present approach we obtain the following corollary:     
    
\begin{corollary}\label{cor-K-nuc}
Both $\mathfrak{A}$ and $A$ are K-nuclear if all the $A_i$ are nuclear. 
\end{corollary}

Throughout this paper, we employ the following standard notation:
For a Hilbert space $H$, we denote by $\lB ( H)$ the bounded linear operators on $H$ and by $\lK ( H)$ the compact ones on $H$.
For $\rmC^*$-algebras $A$ and $B$, $A \otimes B$ stands for the minimal tensor product.
We use the symbol $\odot$ for algebraic tensor products. 
For a subset $\calS$ of a normed space $X$, we denote by $[\calS]$ the closed linear subspace of $X$ generated by $\calS$.

\section{preliminaries}\setcounter{theorem}{0} \renewcommand{\thetheorem}{\arabic{section}.\arabic{theorem}}

\subsection{$\rmC^*$-correspondneces}\label{ss-cor}
For the theory of Hilbert $\rmC^*$-modules, we refer to Lance's book \cite{Lance-95}.
Let $A$ and $B$ be $\rmC^*$-algebras.
An \emph{$A$-$B$ $\rmC^*$-correspondence} is a pair $(X, \pi_X)$ such that $X$ is a Hilbert $B$-module and $\pi_X$ is a $*$-homomorphism from $A$ into the $\rmC^*$-algebra $\lL_B (X)$ of right $B$-linear adjointable operators on $X$.
We denote by $\lK_B (X)$ the $\rmC^*$-ideal of $\lL_B (X)$ generated by ``rank one operators'' $\theta_{\xi, \eta}$, $\xi, \eta \in X$
defined by $\theta_{\xi, \eta } ( \zeta ):= \xi \i<\eta, \zeta >$.
The \emph{identity $\rmC^*$-correspondence over $A$} is the pair $(A, \lambda_A)$, where $A$ is equipped with the $A$-valued inner product $\i< x, y > = x^* y$ for $x, y \in A$ and $\lambda_A : A \to \lL_A ( A)$ is defined by the left multiplication.
It is known that $\lL_A (A)$ is naturally isomorphic to the multiplier algebra $\calM (A)$ of $A$.

\medskip
We use the following two notions of tensor products for Hilbert $\rmC^*$-modules.
Let $X$ be a Hilbert $B$-module and $(Y, \phi )$ be a $B$-$C$ $\rmC^*$-correspondence.
We denote by $X \otimes_{\phi} Y$ the \emph{interior tensor product} of $X$ and $(Y, \phi )$,
which is given by separation and completion of $X \odot Y$ with respect to the $C$-valued inner product $\i< \xi \otimes \eta, \xi' \otimes \eta' > = \i<{ \eta, \phi ( \i< \xi, \xi '> ) \eta }>$.
There is a canonical map $\lL_B (X) \to \lL_C ( X\otimes_\phi Y)$ sending $T$ to the operator $T \otimes_\phi 1_Y : \xi \otimes \eta \mapsto ( T \xi )\otimes \eta$.
For a given $*$-homomorphism $\pi_X : A \to \lL_B (X)$ we define $\pi_X \otimes_\phi 1_Y : A \to \lL_C (X \otimes_\phi Y)$ by $(\pi_X \otimes_\phi 1_Y ) (a )  = \pi_X (a) \otimes_\phi 1_Y$.
When no confusion may arise, we use the notations $X \otimes_B Y$ and $\pi_X \otimes_B 1_Y$ for short.

For a Hilbert $D$-module $Z$, we denote by $X \otimes Z$ the \emph{exterior tensor product} of $X$ and $Y$, which is the completion of $X \odot Y$ with respect to the $B \otimes D$-valued inner product $\i< \xi \otimes \zeta, \xi' \otimes \zeta' > = \i< \xi, \xi '> \otimes \i< \zeta, \zeta '>$.
When $(Z, \pi_Z)$ is a $C$-$D$ $\rmC^*$-correspondence, there is a natural $*$-homomorphism $\pi_X \otimes \pi_Z : A \otimes C \to \lL_{B \otimes D} ( X \otimes Z)$ so that $(X \otimes Z, \pi_X \otimes \pi_Z)$ is an $A\otimes C$-$B \otimes D$ $\rmC^*$-correspondence.

\medskip
Let $B \subset A$ be a quasi-unital inclusion of $\rmC^*$-algebras (i.e., $ BAB$ is norm-dense in $A$) with conditional expectation $E : A \to B$.
We denote by $L^2 (A, E)$ the Hilbert $B$-module given by separation and completion of $A$ with respect to the $B$-valued inner product $\i< x, y > = E ( x^* y )$ for $x, y \in A$, and by $\pi_E : A \to \lL_B ( L^2 ( A, E))$ the $*$-homomorphism induced from the left multiplication.
The conditional expectation $E$ is said to be {\it nondegenerate} if $\pi_E$ is faithful.
We denote by $1_A$ the unit of the multiplier algebra of $A$.
Since the inclusion $B \subset A$ is quasi-unital,
$B$ contains an approximate unit for $A$.
In particular, $A$ is unital if and only if so is $B$, and they should have a common unit.
Thus, we can uniquely extend $E$ to a conditional expectation $\widetilde{E} : A + \lC 1_A \to B + \lC 1_A$.
Let $\xi_E$ be the vector in $L^2 (A + \lC 1_A, \widetilde{E} )$ corresponding to $1_A$.
We always identify $L^2 (A, E)$ with $[ \pi_{\widetilde{E}} (A) \xi_E ] \subset L^2 (A + \lC 1_A, \widetilde{E})$ and
call the triple $(L^2 (A, E), \pi_E, \xi_E)$ the \emph{GNS-representation} associated with the conditional expectation $E$.
Notice that $\xi_E$ need not to be in $L^2 (A, E)$ when $A$ is non-unital.

\subsection{$KK$-theory}\label{ss-KK} 
Throughout this subsection, all $\rmC^*$-algebras are assumed to be separable for simplicity.
We refer the reader to \cite{Blackadar-book} for $KK$-theory.

\begin{definition}
For (trivially graded) $\rmC^*$-algebras $A$ and $B$, a {\it Kasparov $A$-$B$ bimodule} is a triple $(X, \phi, F )$ such that $X$ is a countably generated graded Hilbert $B$-module, $\phi :A \to \lL_B ( X)$ is a $*$-homomorphism of degree 0,
and $F \in \lL_B (X)$ is of degree 1 and satisfies the following condition:
\begin{itemize}
\item $[F, \phi (a)] \in \lK_B (X)$ for $a\in A$,
\item $(F-F^*) \phi (a) \in \lK_B (X)$ for $a\in A$,
\item $(1 - F^2) \phi (a) \in \lK_B (X)$ for $a \in A$.
\end{itemize}
When $[F, \phi (a)]=(F-F^*) \phi (a)=(1 - F^2) \phi (a)=0 $ holds for every $a\in A$,
we say that $(X, \phi, F )$ is \emph{degenerate}.
We denote by $\lE (A,B)$ and $\lD (A, B)$ the corrections of Kasparov $A$-$B$ bimodules and degenerate ones, respectively.
\end{definition}

We say that two Kasparov $A$-$B$ bimodules $(X, \phi, F)$ and $(Y, \psi, G)$ are \emph{unitarily equivalent}, denoted by $(X, \phi, F) \cong (Y, \psi, G)$, if there exists a unitary $U \in \lL (X , Y)$ of degree $0$ such that $\psi = \Ad U \circ \phi$ and $G= UFU^*$.

For any Hilbert $B$-module $X$, we set $IX:=C([0,1]) \otimes X$.
In particular, we set $IB = C([0, 1] ) \otimes B$.
For each $t \in [0, 1]$ we still denote by $t$ the surjectiove $*$-homomorphism $IB \cong C ( [0,1 ], B) \ni f \mapsto f (t ) \in B$.
Note that we have a natural isomorphism $IX \otimes_t B \cong X$ for every $t \in [0,1]$.

\begin{definition}
Two Kasparov $A$-$B$ bimodules $(X_0, \phi_0, F_0 )$ and $(X_1, \phi_1, F_1 )$ are said to be \emph{homotopic} if there exists a Kasparov $A$-$IB$ bimodule $(Y, \psi, G)$ such that $(Y \otimes_t  B, \psi \otimes_t 1_B, G \otimes_t 1_B)  \cong (X_t , \phi_t, F_t )$ for $t = 0,1$.
The $KK$-group $KK(A,B)$ is the set of homotopy equivalence classes of all Kasparov $A$-$B$ bimodules.
\end{definition}

The next technical lemma will be used later.
\begin{lemma}\label{lem-homotopy}
Let $P, Q$ and $R$ be separable $\rmC^*$-algebras and let $(X, \psi_i, F) \in \lE ( Q, R)$ be given for $i =0, 1$.
Suppose that there exist a surjective $*$-homomorphism $\pi : P \to Q$ and a family of Kasparov $P$-$R$ bimodules $(X, \phi_t, F)$ for $t \in [ 0,1]$ satisfying
\begin{itemize}
\item[(i)] the mapping $[ 0,1 ] \ni t \mapsto \phi_t (  a )$ is strictly continuous for each $a \in P$;
\item[(ii)] $\phi_t $ factors through $\pi : P \to Q$ for every $t \in [ 0,1 ]$;
\item[(iii)] $\phi_i = \psi_i \circ \pi$ holds for $i  = 0,1$.
\end{itemize}
Then, $(X, \psi_0, F)$ and $(X, \psi_1, F)$ are homotopic. 
\end{lemma}
\begin{proof}
By assumption, there exists a $*$-homomorphism $\phi : P \to \lL_{IR} ( IX)$ such that $(IX, \phi, F \otimes 1_{C [0,1] }) \in \lE ( P, IR)$ and $\phi \otimes_t 1_R  = \phi_t$ for $t \in [ 0,1]$.
Since one has $\| \phi ( a) \| = \sup_{0 \leq t \leq 1} \| \phi_t ( a) \|  \leq \| \pi ( a) \|$ for $ a\in P$,
there exists $\psi : Q \to \lL_{IR} (IX)$ such that $\phi = \psi \circ \pi$.
We then have $(IX, \psi, F \otimes 1_{C ([0,1] )} ) \in \lE ( Q, IR)$ and the evaluations of this Kasparov bimodule at endpoints are exactly $(X, \psi_i, F)$, $i =0,1$.
\end{proof}

The $KK$-group becomes an additive group in the following way:
For $\alpha, \beta \in KK (A, B)$ implemented by $(X, \phi, F), (Y, \psi, G)$, respectively, $\alpha + \beta$ is the element implemented by $(X \oplus Y, \phi \oplus \psi, F \oplus G)$.
All degenerate Kasparov bimodules are homotopic to the trivial bimodule $0=(0,0,0)$ and define the zero element in $KK(A, B)$.
Let $X_0$ and $X_1$ be the even and odd parts of $X$ so that $X = X_0 \oplus X_1$ and let $-X$ be the graded Hilbert $B$-module with the even part $X_1$ and the odd part $X_0$.
The inverse of $\alpha$ is implemented by $(-X, \Ad U \circ \phi, U F U^*)$, where $U :  X \to -X$ is the natural unitary. 

For any $*$-homomorphism $\phi : A \to B$, we have $(B \oplus 0, \phi \oplus 0, 0 ) \in \lE (A, B)$ and still denote by $\phi$ the corresponding element in $KK (A, B)$. 

For $\alpha \in  KK(A,B)$ and $\gamma  \in KK (B, C)$, the \emph{Kasparov product} of $\alpha$ and $\gamma$ is denoted by $\gamma \circ \alpha$ (or $\alpha \otimes_B \gamma$).
When one of $\alpha$ and $\beta$ comes from a $*$-homomorphism, the construction of the Kasparov product is very simple (and we will use Kasparov products only in these special cases).
Indeed, if $\gamma$ comes from a $*$-homomorphism $\gamma : B \to C$ with $[ \gamma (B) C] =C$ and $\alpha$ is implemented by $(X, \phi, F)$,
then the Kasparov product $\gamma \circ \alpha$ is implemented by $( X \otimes_\gamma C, \phi \otimes_\gamma 1_C, F \otimes_\gamma 1_C)$.
Similarly, when $\alpha$ is a $*$-homomorphism from $A$ into $B$ and $\gamma$ is implemented by $(Y, \psi, G)$ with $[\psi ( B) Y] = Y$,
the Kasparov product $\gamma \circ \alpha$ is implemented by $(Y, \psi \circ \alpha, G)$.

\begin{definition}
An element $\alpha \in KK (A, B)$ is said to be a $KK$-{\it equivalence} if there exists $\beta \in KK( B, A)$ such that $\id_A =\beta \circ \alpha $ and $\id_{B} = \alpha \circ \beta$.
In this case, $A$ and $B$ are said to be \emph{$KK$-equivalent}.
\end{definition}

Note that $KK$-equivalence between $A$ and $B$ implies $KK (A, C) \cong KK (B, C)$ and $KK (C, A) \cong KK ( C, B)$ for any separable $\rmC^*$-algebra $C$.

Finally, we recall the notion of $K$-nuclearity in the sense of Skandalis \cite{Skandalis}.

\begin{theorem}[{\cite[Theoreme 1.5]{Skandalis}}] \label{thm-nuclear}
Let $A$ and $B$ be separable $\rmC^*$-algebras and let $\pi : A \to \lB (H)$ be a faithful and essential representation on a separable Hilbert space $H$.
For a given $A$-$B$ $\rmC^*$-correspondence $(X, \sigma)$ with $X$ countably generated,
the following are equivalent:
\begin{itemize}
\item[(i)] For any unit vector $\xi \in X$ the c.c.p.~(completely contractive positive) map $A \ni a \mapsto \i< \xi, \sigma ( a) \xi > \in B$ is nuclear.
\item[(ii)] For any $x \in \lK_B ( X)$ of norm 1, the c.c.p.~map $A \ni a \mapsto x^* \sigma (a) x \in \lK_B (X)$ is nuclear.
\item[(iii)] There exists a sequence of isometries $V_n \in \lL_B (X, H \otimes B)$ such that $\sigma ( a) - V_n^* (\pi ( a ) \otimes 1_A) V_n \in \lK_B (X)$ and $\lim_{n \to \infty} \| \sigma ( a) - V_n^* (\pi ( a ) \otimes 1_A) V_n \| =0$ for all $a \in A$.
\end{itemize}
When any of these three conditions holds, we say that $(X, \sigma)$ is \emph{nuclear}.
\end{theorem}

Note that any $\rmC^*$-correspondence of the form $(X \otimes_B Y, \pi_X \otimes_B 1_Y)$ is nuclear whenever $B$ is nuclear (see e.g.~\cite[Remark 2.11]{Hasegawa}).

\begin{definition}
A separable $\rmC^*$-algebra $A$ is said to be \emph{$K$-nuclear} if $\id_A$ in $KK (A, A)$ is implemented by a Kasparov bimodule $(X, \phi, F)$ such that $(X, \phi)$ is nuclear.
\end{definition}

\subsection{Amalgamated free products}

Let $\{ B \subset A_i \}_{ i \in \calI}$ be a family of inclusions of $\rmC^*$-algebras.
Recall that the \emph{full amalgamated free product} of $\{ A_i \}_{ i \in \calI}$ over $B$ is
a $\rmC^*$-algebra $\frakA$ generated by the images of injective $*$-homomorphisms $f_i : A_i \hookrightarrow \frakA$ such that $f_i |_B = f_j |_B$ for $i,j \in \calI$
and satisfying the following universal property:
for any $\rmC^*$-algebra $C$ and $*$-homomorphisms $\pi_i : A_i \to C$ satisfying $\pi_i |_B = \pi_j|_B$ for $i,j \in \calI$,
there exists a unique $*$-homomorphism $\bigstar_{ i \in \calI} \pi_i : \frakA \to C$ such that $(\bigstar_{ i \in \calI} \pi_i  )\circ f_i = \pi_i$ for $i \in \calI$.
Since the full amalgamated free product is unique up to isomorphism, we denote it by $\bigstar_{B, i \in \calI} A_i$.
We identify $A_i$ with $f_i (A_i)$ so that $A_i \subset \bigstar_{B, i \in \calI} A_i$ for every $i \in \calI$.

\medskip
Further assume that, the inclusion $B \subset A_i$ is quasi-unital and there exists a nondegenerate conditional expectation $E^{A_i}_B : A_i \to B$ for each $i \in \calI$.
In \cite{Voiculescu}, Voiculescu introduced reduced amalgamated free products of unital inclusions of $\rmC^*$-algebras with conditional expectations.
To reduce Theorem \ref{thm-KK} to the case when $\calI = \{ 1,2\}$, we need to extend Voiculescu's definition to quasi-unital inclusions.
For any $m \in \lN$, set $\calI_m:= \{ \iota : \{ 1, \dots, m \} \to \calI \mid \iota (k) \neq \iota (k+1) \text{ for }1 \leq k \leq m-1 \}$.
Recall that the \emph{reduced amalgamated free product} of $\{ (A_i , E^{A_i}_B  ) \}_{ i \in \calI}$ is a pair $(A, E)$ such that
\begin{itemize}
\item  $A$ is a $\rmC^*$-algebra generated by the images of injective $*$-homomorphisms $g_i : A_i \hookrightarrow A$ such that $g_i |_B = g_j |_B$ for $i, j \in \calI$;
\item $E$ is a nondegenerate conditional expectation from $A$ onto $g_i ( B)$ (independent of $i$);
\item one has $E ( g_{\iota (1)} (x_1) g_{\iota (2)} ( x_2) \cdots g_{\iota (m)} (x_m) ) =0$ for any $m \geq 1$, $\iota \in \calI_m$, and $x_k \in \ker E^{A_{\iota (k)}}_{B}$ for $1 \leq k \leq m$.
\end{itemize}
We will also identify $A_i$ with $g_i (A_i)$ for every $i \in \calI$.
Since the pair $(A, E)$ is determined by the above three conditions, we will write $\bigstar_{B, \calI} (A_i, E^{A_i}_B) := (A, E)$.
Clearly, we have a canonical surjection $\lambda : \frakA \to A$ satisfying that $\lambda \circ f_i = g_i$ for every $i \in \calI$.

\medskip
We recall the construction of $(A, E)$ (\cite{Voiculescu}).
Let $(X_i, \pi_{X_i}, \xi_i )$ be the GNS-representation associated with $E^{A_i}_B$ (see \S\S \ref{ss-cor}) and set $A_i^\circ:= \ker E^{A_i}_B$ and $X_i^\circ = X_i \ominus \xi_i B = [ \pi_{X_i} ( A_i^\circ ) \xi_i ]$ for $i \in \calI$.
Recall that the free product of $\{  (X_i, \xi_i ) \}_{i \in \calI}$ is the Hilbert $B$-module $(X, \xi_0 )$ defined to be
$\xi_0 B \oplus \bigoplus_{m \geq 1} \bigoplus_{ \iota \in \calI _m } X_{\iota ( 1)}^\circ \otimes_B \cdots \otimes_B X_{\iota (m)}^\circ$,
where we define $\i< \xi_0 b , \xi_0 c > = b^* c$ for $b,c \in B$.
We identify $X_i $ with $\xi_0 B \oplus X_i^\circ$ so that $X_i \subset X$.
For each $i \in \calI$, we consider the following submodules:
\begin{align*}
X (\ell,i) &:=\xi_0 B \oplus \bigoplus_{ m \geq 1}  \us {\bigoplus_{ \iota \in \calI_m}}_{\iota (1) \neq i } X_{\iota (1)}^\circ \otimes_B \cdots  \otimes_B X_{\iota (m )}^\circ, \\
X (r, i) &:=\xi_0 B \oplus \bigoplus_{ m \geq 1} \us{\bigoplus_{ \iota \in \calI_m}}_{\iota (m) \neq i } X_{\iota (1)}^\circ \otimes_B \cdots  \otimes_B X_{\iota (m )}^\circ.
\end{align*}
Then, there is a natural unitary $V_i : X \cong X_i \otimes_B X (\ell,i )$ (see \cite{Voiculescu}).
We set $g_i:=\Ad V_i \circ (\pi_{X_i} \otimes_B 1_{X (\ell,i )} )$ and $A = \rmC^*( g_i (A_i ) \mid i \in \calI )$.
Then, the $g_i$'s agree with each other on $B$ and the compression map to $\xi_0 B$ gives the desired conditional expectation $E : A \to B$.
Note that the representation $A \subset \lL_B (X)$ with $\xi_0$ is nothing but the GNS-representation associated with $E$, and hence $E$ must be nondegenerate.

The following lemma is probably well-known, but we give its proof for the reader's convenience.
\begin{lemma}\label{lem-red}
Let $\frakA= \bigstar_{B, i \in \calI} A_i$ and $(A, E) = \bigstar_{B, i \in \calI} (A_i, E^{A_i}_B )$ be as above.
Let $C$ be a $\rmC^*$-algebra and $(Z, \pi_Z)$ be an $\frakA$-$C$ $\rmC^*$-correspondence.
Suppose that for each $i \in \calI$ there exists a subset $\calS_i \subset A_i^\circ$ generating $A_i^\circ$ as a normed space,
and there exists a cyclic subspace $\Gamma \subset Z$ (i.e., $[ \pi_Z (\frakA ) \Gamma C] = Z$ holds) which satisfies the freeness condition: for any $m \in \lN$, $\iota \in \calI_m$, $x_k \in \calS_{\iota (k)}$ for $1 \leq k \leq m$ and $\xi, \eta \in \Gamma$, one has $\i< \xi, \pi_Z ( x_1 x_2 \cdots x_m ) \eta > =0$. Then, $\pi_Z$ factors through $\lambda : \frakA \to A$.
\end{lemma}
\begin{proof}
We will show that $\ker \lambda \subset \ker \pi_Z$.
Choose and fix $z \in \ker \lambda$ arbitrarily.
By assumption, it suffices to show that $\i< \xi, \pi_Z ( x z y) \eta > =0$ for all $x, y \in \frakA$ and $\xi, \eta \in \Gamma$.
We may assume that $x$ and $y$ are in $\ast$-$\alg ( \bigcup_{i\in \calI } A_i )$.
Take a sequence $z_n$ in $\ast$-$\alg ( \bigcup_{i\in \calI } A_i )$ such that $\lim_{n \to \infty }\| z - z_n \| = 0$.
For each $n \geq 1$ there exists $b_n \in B$ such that $x z_n y - b_n$ is a sum of some elements of the form $x_1 \cdots x_m$ for some $m \geq 1$, $\iota \in \calI_m$ and $x_k \in A_{\iota (k)}^\circ$ for $1 \leq k \leq m$ so that $b_n = E ( x z_n y)$.
The assumption on $\Gamma$ implies that $\i< \xi, \pi_Z ( x z_n y - b_n ) \eta > =0$, and hence we have $\| \i< \xi, \pi_Z ( x z y )  \eta > \| =\lim_{n \to \infty } \| \i< \xi, \pi_Z ( b_n ) \eta > \| \leq \limsup_{n \to \infty } \| \xi \| \| \eta \| \| E ( x z_n y ) \| = 0$.
 \end{proof}
\section{Proof}

\subsection{Case of two free components}
We first deal with the case when $\calI = \{ 1, 2\}$.
Let $(A, E) =(A_1, E^{A_1}_B) \star_B (A_2, E^{A_2}_B )$, $\frakA=A_1 \star_B A_2$ and $\lambda : \frakA \to A$ be as in Theorem \ref{thm-KK}.
As in the previous section, let $(X, \pi_X, \xi_0 )$ be the GNS-representation associated with $E$ and identify $X_i :=L^2(A_i, E^{A_i}_B)$ with $\xi_0 B \oplus X_i^\circ$ for $i = 1,2$.
Let $E_{A_i} : A \to A_i$ be the canonical conditional expectation given by the compression map to $X_i$ and let $(Y_i, \pi_{Y_i}, \eta_i )$ be the GNS-representation associated with $E_{A_i}$ for $i =1, 2$.
Note that any vector of the form $\xi_0 \otimes a $ with $a \in A_i$ sits in $X (r, i ) \otimes_B A_i$ for each $i \in \calI$ thanks to the assumption that $B \subset A$ is quasi-unital.
The following lemma can be shown in the same manner as \cite[Lemma 3.1]{Vergnioux}, but we give its proof for the reader's convenience.

\begin{lemma}
The exists a unitary $S_i:   X(r, i) \otimes_B A_i \to Y_i$ satisfying that $S_i (\xi_0 \otimes y ) = \eta_i y$ and $ S_i ( x_1 \cdots x_m \xi_0 \otimes y ) = x_1 \cdots x_m  \eta_i y  $ for all $y \in A_i$ and $m \in \lN $, $\iota \in \calI_m$ with $\iota ( m) \neq i $, and $x_k \in A_{\iota (k)}^\circ$ for $1 \leq k \leq m$.
\end{lemma}
\begin{proof}
Note that if $S_i$ is bounded, then it must be surjective.
Thus, it suffices to show that $S_i$ is an isometry.
By the polarization trick, we only have to verify that $E_{A_i} ( x^*x) = E ( x^* x)$ for all $x = x_1 \cdots x_m$ with $m \in \lN$, $\iota \in \calI_m$, $\iota (m) \neq i$ and $x_k \in A_{ \iota (k)}^\circ$ for $1 \leq k \leq m$.
When $m=1$, this follows from the fact that $E_{A_i}$ is given by the compression to $X_i$.
Assume that we have shown for $k=1, \dots, m$.
For $\iota \in \calI_{m+1}$ with $\iota ( m+1) \neq i$, take $x_k \in A^\circ_{\iota (k)}$ arbitrarily.
If we put $y=x_2 \cdots x_{m+1}$ and $b=E(x_1^* x_1)$, then the induction hypothesis implies that $E_{A_i} (y^*b y ) = E( y^*b y)$.
Thus, we have
$E_{A_i} ( x^*x ) = E_{A_i} ( y^* E ( x_1^*x_1 ) y ) + E_{A_i} ( y^* ( x_1^* x_1 - E ( x_1^* x_1) ) y ) = E_{A_i} ( y^*  by ) =E (y^* by ) = E ( y^* E ( x_1^*x_1 ) y ) + E ( y^* ( x_1^* x_1 - E ( x_1^* x_1) ) y ) = E ( x^*x).$
Hence, we are done.
\end{proof}

Consider two $A$-$\frakA$ $\rmC^*$-correspondences $(Z^+, \pi^+):= ( X \otimes_B \frakA, \pi_X  \otimes_B 1 )$ and $(Z^-, \pi^-):= \bigoplus_{i=1}^2 (Y_i \otimes_{A_i} \frakA, \pi_{Y_i}  \otimes_{A_i} 1)$.
Notice that the vector $\zeta_i := \eta_i \otimes 1_\frakA $ is not necessarily in $Z^-$, but one has $\zeta_i \frakA \subset Z^-$.
Define the isometry $S : Z^+ \to Z^-$ by
\[
\begin{cases}
S_1 \otimes_{A_1} 1 :  X (r, 1)^\circ \otimes_B \frakA \to Y_1^\circ \otimes_{A_1} \frakA ; \\
S_2 \otimes_{A_2} 1 :  X (r, 2 ) \otimes_B \frakA \to Y_2 \otimes_{A_2} \frakA .
\end{cases}
\]
\begin{lemma}[{c.f.~\cite[Theorem 3.3 (2)]{Vergnioux}}]
The operator $S$ satisfies that $\ker S^* = \zeta_1 \frakA$ and $\pi^- (a) S - S \pi^+ (a)$ is compact for all $a \in A$.
Consequently, the triple $(Z^+ \oplus Z^-, \pi^+ \oplus \pi^-, \left[ \begin{smallmatrix} 0 &S^* \\ S & 0 \end{smallmatrix} \right] )$ is an $A$-$\frakA$ Kasparov bimodule.
\end{lemma}
\begin{proof}
The first assertion is obvious.
Thus, it suffices to show $\pi^- (x) S - S \pi^+ (x)$ is compact for all $x \in A_1 \cup A_2$.
In fact, since each $x \in A_1$ enjoys $x X (r, 1)^\circ \subset X ( r, 1)^\circ$ and $x X (r, 2 ) \subset X ( r, 2)$,
one has $\pi^- (x) S = S \pi^+ (x)$ for $x \in A_1$.
If we define $S' :Z^+ \to Z^-$ by $S' \xi_0 \otimes a = \zeta_1 a$ for $a \in \frakA$ and by $S$ on $X^\circ \otimes_B \frakA$,
then $S'$ intertwines the actions of $A_2$ by the above argument.
Since $S$ is a compact perturbation of $S'$, we are done.
\end{proof}

\begin{remark}
Recall that the Bass--Serre tree associated with an amalgamated free product group $G = G_1 \star_H G_2$ is the graph whose vertex and edge sets are given by $\Delta^0= G/ G_1 \sqcup G/G_2$ and $\Delta^1 =G/ H$, respectively.
Consider the unitary representations of $G$ on $\ell^2 (\Delta^0) $ and $ \ell^2 (\Delta^1)$ induced from the action of $G$ on $(\Delta^0, \Delta^1)$.
In \cite{Hasegawa}, we saw that $\rmC^*$-correspondences play a role of unitary representations for groups.
In our theory, the canonical representation $G$ on $\ell^2 (G/H)$ corresponds to $(L^2(A, E) \otimes_B \frakA, \pi_{E}\circ \lambda \otimes_B 1)$ (c.f.~\cite{Hasegawa-2}).
Thus, the $\rmC^*$-correspondences $(Z^+, \pi^+\circ \lambda)$ and $(Z^-, \pi^-\circ \lambda)$ should play a role of the Bass--Serre tree in $\rmC^*$-algebra theory.
Also, the adjoint of $S$ corresponds to the co-isometry of Julg--Valette in \cite{Julg-Valette}.
\end{remark}

Here is the main technical result of this paper.
\begin{theorem}\label{thm-KK2}
With the notation above, let $\alpha$ be the element in $KK (A, \frakA)$ implemented by $(Z^+ \oplus Z^-, \pi^+ \oplus \pi^-, \left[ \begin{smallmatrix} 0 &S^* \\ S & 0 \end{smallmatrix} \right] )$.
Then, we have $\alpha \circ \lambda  + \id_\frakA  = 0$ and $\lambda \circ \alpha  + \id_{A}= 0$.
\end{theorem}

\begin{proof}
We first prove that $ \alpha \circ \lambda + \id_\frakA = 0$ following the proof of \cite[Theorem 3.3 (3)]{Vergnioux}.
Set $\rho^+:=\pi^+ \circ \lambda$ and $\rho^-:=\pi^- \circ \lambda$.
Define the unitary $U : Z^+ \oplus \frakA \to Z^-$ by $S$ on $Z^+$ and by $U ( 0 \oplus a ) :=\zeta_1 a $ for $a \in \frakA$.
Since $ S$ is a compact perturbation of $U$, $\alpha \circ \lambda + \id_\frakA $ is implemented by
$
(( Z^+  \oplus \frakA ) \oplus Z^-, (\rho^+ \oplus \lambda_\frakA ) \oplus \rho^-, \left[ \begin{smallmatrix} 0 &U^* \\ U* & 0 \end{smallmatrix} \right] )$ (see \S\S \ref{ss-KK}).
Take a norm continuous path $(v_t )_{0 \leq t \leq 1}$ of unitaries in $\lM_2 ( \lC)$ such that $v_0 =1$ and $v_1 =\left[ \begin{smallmatrix}0 & 1 \\ 1 & 0 \end{smallmatrix} \right]$.
With the natural identification $\lM_2 ( \lC) \subset  \lM_2 ( \lC) \otimes  \calM ( \frakA)  =\lL_\frakA (  \zeta_1 \frakA \oplus \zeta_2 \frakA)$ 
we define the unitary $u_t \in \lL_\frakA (Z^-)$ by $v_t$ on $\zeta_1 \frakA \oplus \zeta_2 \frakA$ and by the identity operator on $Z^- \ominus  ( \zeta_1 \frakA \oplus \zeta_2 \frakA)$.
Since the restriction of $B$ to $\zeta_1 \frakA \oplus \zeta_2 \frakA $ is just $\lC 1 \otimes B  \subset \lM_2 ( \lC ) \otimes \calM (A)$ with the above identification, 
the family $(u _t )_{ 0\leq t \leq 1}$ forms a norm continuous path of unitaries in $\pi^- (B)' \cap ( \lC 1 + \lK_B ( Z^- ) )$ satisfying that $u_0 = 1$ and $u_1$ switches $\zeta_1 a$ and $\zeta_2 a $ for each $ a\in \frakA$.
Let $\iota_i : A_i \hookrightarrow A$ be the inclusion map for $i = 1,2$.
Since $\pi^- \circ \iota_1$ agrees with $\Ad u_t \circ \pi^- \circ \iota_2$ on $B$,
we have the natural $*$-homomorphism $\phi_t := ( \pi^- \circ \iota_1 ) \star (  \Ad u_t \circ \pi^- \circ \iota_2 ) :\frakA \to \lL_\frakA ( Z^-)$ thanks to the universality of $\frakA$.
Then, the Kasparov bimodules
\[
\left( (Z^+ \oplus \frakA) \oplus Z^- , (\rho^+ \oplus \lambda_\frakA) \oplus \phi_t, \left[ \begin{smallmatrix} 0 &U^* \\ U & 0 \end{smallmatrix} \right] \right) ,\quad t \in [0, 1]
\]
satisfy conditions (i) and (ii) in Lemma \ref{lem-homotopy} (with $P = Q = \frakA $),
and its evaluation at $t =0$ implements $\alpha \circ \lambda + \id_\frakA$.
Thus, we need to show that $(( Z^+  \oplus \frakA ) \oplus Z^-, (\rho^+ \oplus \lambda_\frakA ) \oplus  \phi_1, \left[ \begin{smallmatrix} 0 &U^* \\ U & 0 \end{smallmatrix} \right] )$ is degenerate,
that is, 
\begin{equation}\label{eq}
U ( \rho^+ ( x) \oplus \lambda_\frakA ( x) ) = \phi_t ( x) U  \quad \text{for } x \in \frakA.
\end{equation}
Since $U$ is unitary, we may assume that $x$ is in $A_1 \cup A_2^\circ$.
When $x$ is in $A_1$, the above equation is trivial because $S$ intertwines $\pi^+ (x)$ and $\pi^- (x)$.
Let $S'$ be as in the proof of the previous lemma.
Then, we have $u_1^* U = S'$ on $Z^+$ and $u_1^* U ( 0 \oplus a ) = \zeta_2 a$ for $a \in \frakA$.
Since $S'$ intertwines the actions of $A_2$, we have
$  U ( \pi^+ (x) \oplus \lambda_\frakA ( x) ) = u_1 \pi^- (x) u_1^* U$ for every $x \in A_2$.
Thus we obtain equation (\ref{eq}), and hence
Lemma \ref{lem-homotopy} shows $\alpha \circ \lambda + \id_\frakA = 0$.

\medskip
We next prove that $\lambda \circ  \alpha + \id_{A} = 0$ in $KK (A, A)$.
Note that $\lambda \circ \alpha + \id_{A}$ is implemented by the Kasparov $A$-$A$ bimodule
\[
\left(
( (Z^+ \otimes_\lambda A) \oplus A ) \oplus ( Z^- \otimes_\lambda A ), (( \pi^+ \otimes_\lambda  1_{A}) \oplus \lambda_{A} ) \oplus ( \pi^- \otimes_\lambda 1_{A} ) , \left[ \begin{smallmatrix} 0 & U \otimes_{\lambda} 1_{A} \\ U^*\otimes _{\lambda}1_{A} & 0 \end{smallmatrix} \right]
\right)
\]
(see \S\S \ref{ss-KK}).
We observe that the family of Kasparov $\frakA$-$A$ bimodules
\[
\left ( ( ( Z^+ \otimes_\lambda {A} ) \oplus A ) \oplus  ( Z^- \otimes_\lambda A ) , ( ( \rho^+\otimes_\lambda 1_{A} ) \oplus \lambda ) \oplus ( \phi_t  \otimes_\lambda 1_{A}) , \left[ \begin{smallmatrix} 0 & U \otimes _{\lambda} 1 \\ U^* \otimes _{\lambda} 1& 0 \end{smallmatrix} \right] \right),\quad
t \in [ 0,1]
\]
satisfy conditions (i) and (ii) in Lemma \ref{lem-homotopy} (with $P =\frakA$ and $Q =A$)
and its evaluations at endpoints implement $ ( \lambda \circ \alpha + \id_{A} ) \circ \lambda$ and $0$.
Thus, by Lemma \ref{lem-homotopy} and the fact that $\pi_X : A \to \lL_B (X)$ is faithful,
it suffices to show that $\phi_t  \otimes_{\pi_X \circ \lambda} 1_X : \frakA \to \lL_B ( Z^- \otimes_{\pi_X \circ \lambda} X)$ factors through $\lambda : \frakA \to A$ for every $t \in [0,1]$.
If we set $\sigma :=\bigoplus_{i=1}^2 \pi_{Y_i} \otimes_{A_i} 1_X: A \to \lL_B ( ( Y_1 \otimes_{A_1} X ) \oplus ( Y_2 \otimes_{A_2} X) )$ and $w_t:=u_t \otimes_{\pi_X \circ \lambda} 1_X \in \lL_B (( Y_1 \otimes_{A_1} X)  \oplus ( Y_2 \otimes_{A_2} X ) )$,
then $\phi_t \otimes_{\pi_X \circ \lambda } 1_X$ coincides with $\psi_t :=( \sigma \circ \iota_1 ) \star ( \Ad w_t \circ \sigma \circ \iota_2)$.
Note that $\psi_0= \sigma \circ \lambda $ and $\psi_1 \cong(  \rho^+ \otimes_B 1_X )  \oplus \pi_X \circ \lambda$ apparently factor through $\lambda$.
Thus, we only have to deal with $0 < t < 1$ and we write $w = w_t$ for short.

For the convenience, 
we identify $X (r, i) \otimes_B X$ with $Y_i \otimes_{A_i} X$ via $S_i \otimes_{A_i} 1$ as \emph{right} $B$-modules.
To distinguish between vectors in $X (r, 1) \otimes_B X$ and $X (r, 2) \otimes_B X$,
we use the symbols $\dotimes$ and $\ddotimes$ as markers in such a way that,
for $ \zeta \in X$ we denote by $ \xi_0  \dotimes \zeta \in X (r, 1) \dotimes_B X$
and $\xi_0 \ddotimes \zeta \in X ( r ,2) \ddotimes_B X$ the vectors corresponding to $\eta_1 \otimes \zeta$ and $\eta_2 \otimes \zeta$, respectively.
Thanks to Lemma \ref{lem-red}, the proof will be completed by proving the following claim:

\begin{claim}\label{claim}
The subspace $\Gamma := w ( \xi_0 B \dotimes_B X (\ell,1) ) + \xi_0 B  \ddotimes_B X (\ell,2)$ satisfies the assumption of Lemma \ref{lem-red}.
\end{claim}

We first show that $\Gamma$ is cyclic for $\psi_t (\frakA)$.
Let $\Lambda:= [ \psi_t ( \frakA) \Gamma ]$.
We set $\frakX_0 := \xi_0 B$, $\frakX_m:= \bigotimes_{\iota \in \calI_m} X_{\iota (i)}^\circ \otimes_B \cdots \otimes_B X_{\iota (m)}^\circ$, $\frakX_m (\ell,i ) = \frakX_m \cap X (\ell,i)$ and $\frakX_m (r, i ) = \frakX_m \cap X (r, i )$ for $ m \in \lN$ and $i =1,2$.
It suffices to show that, for any $m \in \{ 0 \} \cup \lN $, $\Lambda$ contains
\[
\frakY_m := \left( \bigoplus_{k=0}^m \frakX_k (r, 1) \dotimes_B \frakX_{m-k} \right) \oplus \left(
\bigoplus_{k=0}^m \frakX_k (r, 2) \ddotimes_B \frakX_{m-k} \right). 
\]
We will show this by induction.
When $m=0$, this is trivial because $\frakY_0 = (\xi_0  B \dotimes_B  \xi_0 B )  \oplus ( \xi_0 B \ddotimes_B \xi_0 B ) = w ( \xi_0 B \dotimes_B \xi_0 B  )  + \xi_0 B \ddotimes_B \xi_0 B$.
Suppose that $\Lambda$ contains $\frakY_k$ for $0 \leq k \leq m $.
Since $w$ is equal to $1$ on the complement of $(\xi_0 B \dotimes_B X )  \oplus ( \xi_0 B \ddotimes_B X )$, it is easily seen that
\[
 \left( \bigoplus_{k=2}^{m+1} \frakX_k (r, 1) \dotimes_B \frakX_{m+1-k} \right) \oplus \left(
\bigoplus_{k=2}^{m+1} \frakX_k (r, 2) \ddotimes_B \frakX_{m+1-k} \right) \subset \Lambda. 
\]
Thus, we only have to check that $\Lambda$ contains the following six subspaces:
\begin{align*}
X_2^\circ \dotimes_B \frakX_m,
&& \xi_0 B \dotimes_B \frakX_{m+1} ( \lambda,1),
&& \xi_0 B \dotimes_B \frakX_{m+1} (\ell,2),\\
X_1^\circ \ddotimes_B \frakX_m,
&& \xi_0 B \ddotimes_B \frakX_{m+1} ( \lambda,1),
&& \xi_0 B \ddotimes_B \frakX_{m+1} (\ell,2).
\end{align*}
By assumption of induction, one has $w ( \xi_0  B \dotimes_B \frakX_m ) \subset \frakY_m \subset \Lambda$,
and hence $X_2^\circ \dotimes_B \frakX_m = [ w \sigma ( A_2^\circ ) w w^* ( \xi_0 B \dotimes_B \frakX_m ) ]  \subset \Lambda$.
We also have $X_1^\circ \ddotimes_B \frakX_m = [ \sigma (A_1 ) ( \xi_0 B \ddotimes_B \frakX_m) ] \subset \Lambda$.
We observe that $ w ( \xi_0 B \ddotimes_B \frakX_{m+1} (\ell,1 ) ) = [ w \sigma (A_2^\circ ) w^* w ( \xi_0 B \ddotimes_B \frakX_m (\ell,2) )  ]  \subset \Lambda$ and $w ( \xi_0 B \dotimes_B \frakX_{m+1} (\ell,1) ) \subset \Gamma$ by the definition of $\Gamma$.
Thus, one has
\[
\xi_0 B \dotimes_B \frakX_{m+1} (\ell,1) + \xi_0 B \ddotimes_B \frakX_{m+1} (\ell,1) = w \left( \xi_0 B \dotimes_B \frakX_{m+1} (\ell,1 ) + \xi_0 B \ddotimes_B \frakX_{m+1} (\ell,1) \right) \subset \Lambda.
\]
Finally we obtain that $ \xi_0 B \dotimes_B \frakX_{m+1} (\ell,2) =[ \sigma ( A_1 )( \xi_0 B \dotimes_B \frakX_m (\ell,1) ) ]  \subset \Lambda$ and $\xi_0 B \ddotimes_B \frakX_{m+1} (\ell,2) \subset \Gamma$ by the definition of $\Gamma$ again.
Therefore, by induction, it follows that $\Gamma$ is cyclic for $\psi_t (\frakA )$.

\medskip
We next show that $\Gamma=w ( \xi_0 B \dotimes_B X (\ell,1) ) + \xi_0 B  \ddotimes_B X (\ell,2)$ satisfies the freeness condition.
Let $\Gamma_1 :=\xi_0 B \dotimes_B X (\ell,2)^\circ $ and $\Gamma_2 := w ( \xi_0 B \ddotimes_B X (\ell,1)^\circ )$.
We then claim that the following inclusions hold:
\begin{align}\label{hoge}
\sigma ( A_1^\circ ) \Gamma \subset \Gamma_1 +  X_1^\circ \ddotimes_B X 
\quad \text{and}\quad
w \sigma (A_2^\circ ) w^* \Gamma \subset \Gamma_2 +  X_2^\circ \dotimes_B X.
\end{align}
Indeed, for any $x \in A_1^\circ$
one has
\begin{align*}
\sigma (x ) w ( \xi_0 B \dotimes_B X (\ell,1) )
&\subset \sigma ( x ) ( \xi_0 B \dotimes_B X(\ell,1) +  \xi_0 B \ddotimes_B X (\ell,1) ) \\
&\subset \xi_0 B \dotimes_B X (\ell,2)^\circ + X_1^\circ \ddotimes_B X (\ell,1)\\
&\subset  \Gamma_1 + X_1^\circ \ddotimes_B X
\end{align*}
and $\sigma (x ) ( \xi_0 B \ddotimes_B X (\ell,2) ) \subset X_1^\circ \ddotimes_B X (\ell,2)$.
Similarly, for any $y \in A_2^\circ$ one has
\begin{align*}
w \sigma (y) w^* w ( \xi_0 B \dotimes_B X (\ell,1) ) \subset w (X_2^\circ \dotimes_B X (\ell,1) ) = X_2^\circ \dotimes_B X (\ell,1)
\end{align*}
and
\begin{align*}
w \sigma (y) w^* ( \xi_0 B \ddotimes_B X (\ell,2) )
&\subset w \sigma (y) (\xi_0 B \dotimes_B X (\ell,2) +\xi_0 B \ddotimes_B X (\ell,2) ) \\
&\subset w (X_2^\circ \dotimes_B X (\ell,2)) + w ( \xi_0 B \ddotimes_B X (\ell,1)^\circ )\\
&=  X_2^\circ \dotimes_B X (\ell,2) + \Gamma_2.
\end{align*}
The subspaces on the right hand side in both equations (\ref{hoge}) are apparently orthogonal to $\Gamma$, and one can easily verify that $\sigma (A_1^\circ) \Gamma_2 \subset \Gamma_1 + X_1^\circ \ddotimes_B X $ and $w \sigma (A_2^\circ )w^* \Gamma_1 \subset \Gamma_2 + X_2^\circ \dotimes_B X $.
Since $w=1$ on the complement of $( \xi_0 B \dotimes_B X ) \oplus ( \xi_0 B \ddotimes_B X )$,
the above observations show that for any $x_k \in A_{\iota (k)}$ for $k=1,\dots, m$ with $\iota \in \calI_m$,
the subspace $\psi_t (x_1 \cdots x_m ) \Gamma$ is contained in $\Gamma_{1} + X ( \ell, 2 )^\circ \dotimes_B X + X ( \ell, 2 )^\circ \ddotimes_B X $ when $ \iota (1) =1$, and in $\Gamma_{2} + X ( \ell, 1 )^\circ \dotimes_B X + X ( \ell, 1 )^\circ \ddotimes_B X $ when $\iota (1) =2$. 
This implies that $\Gamma$ satisfies the freeness condition.
\end{proof}

\subsection{Case of countably many free components}
Let $\calI$ be a general countable set and let $\frakA = \bigstar_{B , i \in \calI } A_i$ and $(A, E) = \bigstar_{B, i \in \calI } (A_i, E^{A_i}_B )$ be as in Theorem \ref{thm-KK}.
We set $c_0 := c_0 ( \calI)$ and $\calK := \lK ( \ell^2 ( \calI ))$.

\begin{proposition}\label{prop-infty-2}
With the notation above, there exist nondegenerate conditional expectations $\sum_i E^{A_i}_B : \sum_i A_i \to c_0  \otimes B$ and $E_{c_0} \otimes \id_B : \calK  \otimes B \to c_0  \otimes B$.
If we set $\widetilde{\frakA} :=( \sum_i A_i ) \star_{c_0  \otimes B} ( \calK  \otimes B )$ and $(\widetilde{A}, \widetilde{E}):=(\sum_i A_i , \sum_i E^{A_i}_B ) \star_{c_0 \otimes B} ( \calK  \otimes B, E_{c_0} \otimes \id_B )$,
then there exist isomorphisms $\pi : \widetilde{\frakA}  \to \calK   \otimes \frakA$ and $\pi_\red : \widetilde{A} \to \calK  \otimes A$ such that 
the following diagram
\[
\begin{CD}
\widetilde{\frakA} @>\pi >> \calK \otimes \frakA  \\
@V \widetilde{\lambda} VV @VV \id_\calK \otimes \lambda V \\
\widetilde{A}  @>\pi_\red >> \calK \otimes A 
\end{CD}
\]
commutes, where $\widetilde{\lambda}$ is the canonical surjection.
\end{proposition}

\begin{proof}
Since the proof in the case when $\calI$ is finite is essentially same as (and easier than) the case when $\calI = \lN$,
we may and do assume that $\calI = \lN$.
Let $\{ e_{ij} \}_{i,j \geq 1}$ be the system of matrix units for the canonical basis $\{ \delta_i \}_{i \geq 1}$ of $\ell^2 ( \lN )$, and set $f_n :=e_{nn}$.
We realize $\sum_{n \geq 1} A_n$ and $c_0 \otimes B$ inside $\calK \otimes A$ as 
\[
\sum_{n \geq 1} A_n = \rmC^* \{ f_n \otimes a \mid n \geq 1, a \in A_n \}\quad\text{and}\quad c_0 \otimes B = \rmC^* \{ f_n \otimes b \mid n \geq 1, b \in B \}.
\]
Consider two conditional expectations $\sum_n E^{A_n}_B : \sum_n A_n \to c_0 \otimes B$ and $E_{c_0} \otimes \id_B :  \calK \otimes B \to c_0 \otimes B$ defined by $( \sum_n E^{A_n}_B )  ( f_i \otimes a ) = f_i \otimes E^{A_i}_B (a)$ and $( E_{c_0} \otimes \id_B )( e_{ij} \otimes b  ) = \delta_{i,j} f_i \otimes b$ for $i,j \in \lN$, $a \in A_i$ and $b \in B$.
Set $\widetilde{\frakA}:= (\sum_{n \geq 1} A_n ) \star_{c_0 \otimes B} ( \calK \otimes B) $ and $(\widetilde{A}, \widetilde{E}) := (\sum_{n \geq 1} A_n, \sum_n E^{A_n}_B ) \star_{c_0 \otimes B} (\calK \otimes B, E_{c_0} \otimes \id_B )$ and let $\widetilde{\lambda} : \widetilde{\frakA} \to \widetilde{A}$ be the canonical surjection.

The inclusion maps $ \sum_n A_n \hookrightarrow \calK \otimes \frakA$ and $\lK \otimes B \hookrightarrow \calK \otimes \frakA$ induce a $*$-homomorphism $\pi :  \widetilde{\frakA} \to \calK \otimes \frakA$.
For any $n, i,j  \in \lN$, $a \in A_n$ and $b ,c \in B$, one has $e_{ij} \otimes b a c = \pi (e_{i n }  \otimes b )  \pi (f_n \otimes a ) \pi ( e_{n j } \otimes c ) \in \pi ( \widetilde{\frakA } )$.  
Since $[ B A_n B ] =A_n$ holds, $\pi$ is surjective.
Define $ \sigma_n : A_n \to \widetilde{\frakA}$ by $\sigma_n (  ba c ) =(e_{1n} \otimes b ) (f_n \otimes a ) (e_{n1} \otimes c)$ for $a \in A_n$ and $b,c \in B$.
We then obtain $\sigma  = \bigstar_{n \geq 1} \sigma_n: \frakA \to \widetilde{\frakA}$.
Define $\widetilde{\sigma}: \calK \otimes \frakA \to \widetilde{\frakA}$ by $\widetilde{\sigma} ( e_{ij} \otimes  b a c ) = ( e_{i 1} \otimes b )  \sigma (a ) ( e_{1 j } \otimes c)$ for $ a\in A$ and $i,j \geq 1$.
Then, it is easy to see that $\widetilde{\sigma} \circ \pi = \id_{\widetilde{\frakA}}$, and hence $\pi$ is bijective.

We next see that $( \id_\calK \otimes \lambda) \circ \pi$ and $\widetilde{\lambda} \circ \widetilde{\sigma}$ factor through $\widetilde{\lambda}: \widetilde{\frakA} \to \widetilde{A}$ and $\id_{\calK} \otimes \lambda : \calK \otimes \frakA \to \calK \otimes A $, respectively.
Note that $( \sum_n A_n )^\circ$ and $( \calK \otimes B )^\circ $ are generated by $\calS_1 = \{ f_n \otimes a \mid a \in A_n^\circ, n \geq 1 \}$ and $\calS_2 = \{ e_{ij} \otimes b \mid i, j \geq 1, i \neq j, b \in B \}$ as normed spaces, respectively.
We represent $\calK \otimes A $ on the Hilbert $B$-module $ \ell^2 ( \lN )  \otimes L^2 (A, E) $ faithfully
and show that $ \lC \delta_1 \otimes \xi_E B$ satisfies the assumption of Lemma \ref{lem-red} for $\calS_1$ and $\calS_2$.
Let $m \geq 1$, $x_1, \dots, x_m \in \calS_1$ and $y_1 , \dots, y_m \in \calS_2$ be arbitrarily given.
If $x_1 y_2 \cdots x_m y_m$ is a nonzero element,
then it should be of the form
\[
(f_{\iota ( 1)} \otimes a_1  ) ( e_{\iota (1) \iota (2)} \otimes b_1 )  \cdots (f_{\iota ( m)} \otimes a_m  ) ( e_{\iota (m) \iota (m +1) } \otimes b_m ) = e_{\iota (1) \iota (m+1)} \otimes ( a_1 b_1 \cdots a_m b_m )
\]
for some $\iota \in \calI_{m+1}$, $a_k \in A_{\iota (k)}^\circ$ and $b_k \in B$ for $1 \leq k  \leq m$.
Clearly, this implies that $(x_1 y_1 \cdots x_m y_m) ( \delta_1 \otimes \xi_E B) \perp \delta_1 \otimes \xi_E B$.
A similar assertion holds for $x_1 y_1 \cdots y_m x_{m +1}$, $y_0 x_1 \cdots x_m y_m$ and $y_0 x_1\cdots y_m x_{m+1}$ for any $x_{m+1} \in \calS_1$ and $y_0 \in \calS_2$.
Therefore, Lemma \ref{lem-red} guarantees that $( \id_\calK \otimes \lambda ) \circ \pi$ factors through $\widetilde{\lambda}$.

Similarly, representing $\widetilde{A}$ on $L^2 (\widetilde{A}, \widetilde{E})$ faithfully we observe that $\xi_{\widetilde{E}} ( c_0 \otimes B)$ satisfies the freeness condition for $\calS_n' :=  B A_n^\circ B$, $n \geq 1$.
Indeed, for $m \geq 1$, $\iota \in \calI_m$, $y_k  \in A_{\iota (k) }^\circ$, and $b_k, c_k \in B$,
we have
\[
\sigma ( (b_1 y_1 c_1 ) \cdots ( b_m y_m c_m ) )
=( e_{1 \iota ( 1)} \otimes b_1 ) ( f_{\iota (1)} \otimes y_1 ) ( e_{\iota ( 1) \iota (2)} \otimes c_1 b_2 ) \cdots (f_{\iota (m) }\otimes y_m ) (e_{\iota (m) 1}  \otimes c_m ),
\]
which belongs to $\ker \widetilde{E}$.
Thus, $\widetilde{\lambda} \circ \sigma : \frakA \to \widetilde{A}$ factors through $\lambda : \frakA \to A$, which implies that $\widetilde{\lambda} \circ \widetilde{\sigma}$ factors through $\id_{\calK} \otimes \lambda : \calK \otimes \frakA \to \calK \otimes A$.
\end{proof}

The following general fact is well-known (see, e.g. \cite[Proposition 17.8.7]{Blackadar-book}).
\begin{proposition}\label{prop-stab}
Let $\calK$ be as above and let $\iota : \calK \hookrightarrow \lB ( \ell^2 ( \calI))$ be the inclusion map.
Fix a minimal projection $e \in \calK$.
For any separable $\rmC^*$-algebras $\calA$ and $\calB$,
the mapping $\lE ( \calA, \calB) \ni ( X, \phi, F) \mapsto ( \calK \otimes X, \lambda_\calK \otimes \phi, 1_\calK \otimes F) \in \lE ( \calK \otimes \calA, \calK \otimes \calB)$ induces an isomorphism $\tau : KK (\calA, \calB ) \to KK ( \calK \otimes \calA, \calK \otimes \calB)$.
The inverse of $\tau$ is given by the mapping $\lE ( \calK \otimes \calA, \calK \otimes \calB) \ni ( Y, \psi, G) \mapsto ( Y \otimes_{ \iota \otimes \lambda_\calB} (\ell^2 ( \calI) \otimes \calB),  (\psi \otimes_{ \iota \otimes \lambda_\calB} 1 ) \circ \sigma, G \otimes_{ \iota \otimes \lambda_\calB} 1) \in \lE ( \calA, \calB)$, where $\sigma ( a) = e \otimes a$ for $ a\in \calA$.
\end{proposition}

We are now ready to prove Theorem \ref{thm-KK} and Corollary \ref{cor-K-nuc}.
\begin{proof}[Proof of Theorem \ref{thm-KK} and Corollary \ref{cor-K-nuc}]
We use the notation in the proof of Proposition \ref{prop-infty-2}.
By Theorem \ref{thm-KK2} and Proposition \ref{prop-infty-2}, there exists $\beta \in KK ( \calK \otimes A, \calK \otimes \frakA)$ such that $ \beta \circ (\id_\calK \otimes \lambda)  = \id_{\calK \otimes \frakA}$ and $(\id_\calK \otimes \lambda ) \circ \beta = \id_{\calK \otimes A}$. 
Let $\tau$ be as in Proposition \ref{prop-stab}.
We then have $\id_\frakA = \tau^{-1} ( \id_{\calK \otimes \frakA} ) =\tau^{-1} ( \beta \circ (\id_{\calK} \otimes \lambda )) = \tau^{-1} ( \beta ) \circ \lambda$ and $\id_{A} = \tau^{-1} ( \id_{\calK \otimes A} ) = \tau^{-1} ( (\id_{\calK} \otimes \lambda  ) \circ \beta ) =\lambda \circ \tau^{-1} ( \beta )$. Thus, $\lambda$ gives a $KK$-equivalence.

Moreover, by Theorem \ref{thm-KK2} and Proposition \ref{prop-stab} again, $\tau^{-1} ( \beta)$ is implemented by a Kasparov $A$-$\frakA$ bimodule whose ``$\rmC^*$-correspondence part'' is the direct sum of three $\rmC^*$-correspondences of the form $(Y \otimes_D Z, \pi_Y \otimes_D 1_Z)$, where $D$ is either $c_0 \otimes B$, $\sum_i A_i$ or $\calK \otimes B$.
Thus, if $A_i$ is nuclear for every $i \in \calI$, then $\id_\frakA = \tau^{-1} ( \beta) \circ \lambda$ is also implemented a Kasparov bimodule consisting of a nuclear $\rmC^*$-correspondence (see the remark just after Theorem \ref{thm-nuclear}),
and hence $\frakA$ is $K$-nuclear.
\end{proof}
\begin{remark}
Theorem \ref{thm-KK} generalizes the previous $K$-amenability results for amalgamated free products of amenable discrete (quantum) groups \cite{Julg-Valette} and \cite{Vergnioux}.
However, we should remark that our result does not imply Pimsner's result that $K$-amenability is closed under amalgamated free products.
Similarly, Corollary \ref{cor-K-nuc} does not imply that $K$-nuclearity is closed under amalgamated free products (even for plain free products).
The latter seems a next interesting question in the direction.
\end{remark}

\section{Six-term exact sequences}

Let $(A,E) =(A_1, E^{A_1}_B) \star (A_2, E^{A_2}_B)$ is as in Theorem \ref{thm-KK2}.
We denote by $i_k : B \to A_k$ and $j_k : A_k \to A, k=1,2$ the inclusion maps.
As we mentioned in the introduction, our $KK$-equivalence and $K$-nuclearity results with Thomsen's result \cite{Thomsen}
imply the following:
\begin{corollary}
With the notation above,
there is a cyclic six-term exact sequence
\[
\begin{CD}
K_0(B) @>(i_{1*},i_{2*}) >> K_0(A_1) \oplus K_0 (A_2) @>j_{1*}- j_{2*}>>K_0(  A )  \\
@AAA                                    @.                                       @VVV \\
K_1 (A  ) @<j_{1*}- j_{2*}<< K_1 (A_1 )\oplus K_1 (A_2) @< (i_{1*},i_{2*}) << K_1 (B)
\end{CD}
\]
If $A_1$ and $A_2$ are further assumed to be nuclear,
then for any separable $\rmC^*$-algebras $D$ there is a cyclic exact sequence
\[
\begin{CD}
KK(B,D) @<i_1^*- i_2^*<< KK(A_1,D) \oplus KK (A_2,D) @<j_1^* + j_2^*<< KK( A, D)  \\
@VVV                                    @.                                       @AAA \\
KK(A , SD ) @>j_1^* + j_2^*>> KK(A_1, SD )\oplus KK (A_2, SD) @>i_1^*- i_2^*>> KK( B,SD)
\end{CD}
\]
\end{corollary}

Note that the second exact sequence of $KK$-groups is new even in the full case.
We also obtain the next corollary from Theorem \ref{thm-KK} and \cite{Eliasen}. 

\begin{corollary}
With the notation above, suppose that $B$ is a direct sum of finite dimensional $\rmC^*$-algebras.
Then, for any separable $\rmC^*$-algebra $D$ there are two cyclic exact sequences:
\[
\begin{CD}
KK(D,B) @>(i_{1*},i_{2*}) >> KK(D,A_1) \oplus KK (D, A_2) @>j_{1*}- j_{2*}>>KK( D, A )  \\
@AAA                                    @.                                       @VVV \\
KK(SD,A  ) @<j_{1*}- j_{2*}<< KK(SD,A_1 )\oplus KK (SD, A_2) @<(i_{1*},i_{2*}) <<KK( SD, B)\\
\end{CD}\]
\[ \begin{CD}
KK(B,D) @<i_1^*- i_2^*<< KK(A_1,D) \oplus KK (A_2,D) @<j_1^* + j_2^*<< KK( A, D)  \\
@VVV                                    @.                                       @AAA \\
KK(A , SD ) @>j_1^* + j_2^*>> KK(A_1, SD )\oplus KK (A_2, SD) @>i_1^*- i_2^*>> KK( B,SD)
\end{CD}
\]
\end{corollary}

Finally,
we would like to point out that a similar result holds for HNN extensions.
We refer the reader to \cite{Ueda-05, Ueda-08} for HNN extensions of $\rmC^*$-algebras.
The next corollary follows from ``the $\rmC^*$-version of Proposition 3.1'', Proposition 3.3 and Proposition 4.2 in \cite{Ueda-08} and Theorem \ref{thm-KK}.
\begin{corollary}
Let $B \subset A$ be unital inclusion of separable $\rmC^*$-algebras with nondegenerate conditional expectation $E : A \to B$,
and $\theta : B \to A$ be an injective $*$-homomorphism whose image is the range of a conditional expectation $E_\theta : A \to \theta (B)$.
Then, the full HNN-extension $A \star_B^{\rm univ} \theta$ and the reduced one $(A, E ) \star_B ( \theta, E_\theta )$ are $KK$-equivalent via the canonical surjection, and there is a six-term exact sequence:
\[
\begin{CD}
K_0(B ) @>(\theta_{*}- \iota_{B *}) >> K_0(A) @>\iota_{A *}>>K_0(  (A, E ) \star_B ( \theta, E_\theta ))  \\
@AAA                                    @.                                       @VVV \\
K_1 ((A , E ) \star_B ( \theta, E_\theta) ) @<\iota_{A *}<< K_1 (A ) @<(\theta_{*}- \iota_{B *}) <<K_1( B )
\end{CD}
\]
Here $\iota_B : B \to A$ and $\iota_A: A \to (A ,E ) \star_B ( \theta, E_\theta ) $ are inclusion maps.
Further assume that $A$ is nuclear.
Then, these HNN-extensions are $K$-nuclear.
\end{corollary}
\begin{remark}
Using Proposition \ref{prop-stab} we can generalize the results in this section to amalgamated free products and HNN extensions of countably many $\rmC^*$-algebras and countably many injective $*$-homomorphisms, respectively.
Such generalizations for HNN extensions include Pimsner--Voiculescu's six-term exact sequence for crossed products by free groups (\cite{Pimsner-Voiculescu-80, Pimsner-Voiculescu-82}) as special cases (see also \cite{Ueda-08}).
\end{remark}

\end{document}